\documentclass[12pt,a4]{article}

\usepackage{amsmath,graphicx,amsthm,amssymb,mathrsfs} 

\newcommand{\exn}{{\bf E}}
\newcommand{\pr}{{\bf P}}
\newcommand{\Na}{\mathbb{N}}
\newcommand{\R}{\mathbb{R}}

\newcommand{\Fi}{\mathcal{F}}
\newcommand{\Gi}{\mathcal{G}}
\newcommand{\Hi}{\mathcal{H}}
\newcommand{\Bi}{\mathcal{B}}

\newcommand{\al}{\alpha}

\newcommand{\la}{\lambda}

\newcommand{\si}{\sigma}

\newcommand{\ind}{{\bf 1}}

\newtheorem{theo}{Theorem}

\newtheorem{lemo}{Lemma}

\title{Parisian ruin with random deficit-dependent delays for spectrally negative L\'evy processes}
\author{Duy Phat Nguyen$^1$ and Konstantin Borovkov$^2$ }

\begin{document}
	\maketitle
	
\footnotetext[1]{School of Mathematics and Statistics, The University of Melbourne, Parkville 3010, Australia.  }	
	
\footnotetext[2]{School of Mathematics and Statistics, The University of Melbourne, Parkville 3010, Australia. Corresponding author,  e-mail: borovkov@unimelb.edu.au.}

\begin{abstract}
We consider an interesting   natural  extension to the Parisian ruin problem under the assumption that  the risk reserve dynamics are given by a spectrally negative L\'evy process. The distinctive feature of this extension is that  the distribution of the random implementation delay windows'  lengths can depend  on  the  deficit at the   epochs when the risk reserve process turns negative, starting a new negative excursion. This includes the possibility of an immediate ruin  when the deficit hits a certain subset. In this general setting, we derive a closed-from expression for the Parisian ruin probability and the joint Laplace transform of the Parisian ruin time and the deficit at ruin.

		\smallskip
		{\it Key words and phrases:} Parisian ruin; random delay; spectrally negative L\'evy process; scale function.
		
		\smallskip
		{\em AMS Subject Classification:} 60G51; 60K40.
	\end{abstract}

\section{Introduction and main results}

The concept of Parisian ruin  was first introduced in actuarial risk theory by Dassios and Wu~\cite{DaWu08} in 2008: ``Parisian type ruin will occur if the surplus falls below zero and stays below zero for a continuous time interval of length~$d.$ In some respects, this might be a more appropriate measure of risk than classical ruin as it gives the office some time to put its finances in order." 
The  time period during which the surplus is allowed to remain negative is called implementation delay (or grace) period, often  referred to just as the delay period.

The idea (and the name as well) of such a concept goes back to the so-called Parisian options whose payoffs depend on the lengths of the excursions of the underlying asset prices above or below a flat barrier. For example, 
the owner of a Parisian down-and-out option will lose the option if the underlying asset price drops below  a given  level and stays  constantly below that level for a time interval longer than a given quantity~$d.$ Stopping times of this kind were first considered by Chesney et al.~\cite{ChJeYo97}. 

Over the last decade, analysis of Parisian ruin probabilities and times in different settings has become a popular topic in the literature. First we will mention papers where the delay period length~$d$  was assumed to be {\em  deterministic and fixed\/} (i.e.,   depending neither of the deficit at the beginning of a negative excursion nor on any other quantity, and remaining the same for all negative excursions of the risk reserve process).  

Dassios and Wu~\cite{DaWu08}  derived the Laplace transform of the time until the Parisian ruin and the probability thereof for the classical Cram\'er--Lundberg (CL) model.  Loeffen et al.~\cite{LoCzPa13} derived  an elegant compact formula for  the Parisian ruin probability in the case where the surplus process is modelled by a spectrally negative L\'evy process (SNLP) $X=\{X_t\}_{t\ge 0}$  (whose trajectories may be of unbounded variation), the answer  involving only the scale function of $X$ and the distribution of~$X_d$. Czarna~\cite{Cz16} studied, also in the SNLP framework, Parisian ruin probabilities with an ``ultimate bankruptcy level", meaning that ruin will also occur whenever the deficit reaches a given deterministic negative level. Simpler proofs and further results for that setting were obtained in Czarna and Renaud~\cite{CzRe16}. 

Li et al.~\cite{LiWiWo18} and Lkabous~\cite{Lk19}  studied the concept of Parisian ruin   under the ``hybrid observation scheme" for SNLPs, where the surplus process is monitored discretely at arrival epochs of an independent Poisson process (that can be interpreted as the observation times of the regulatory body), but is continuously observed once the process value drops below zero.   

Lkabous et al.~\cite{LkCzRe17} studied Parisian ruin for a refracted SNLP  model assuming that the premium payment rate is higher when the process is below zero. In Czarna et al.~\cite{CzLiPaZh17},  the joint law of the Parisian ruin time and the number of claims until  that time was derived for the CL~model. A compact formula for the Parisian ruin probability for a spectrally negative Markov additive risk process was obtained  by Zhang and Dong~\cite {ZhDo18}. The probability was expresed in terms of the scale matrix and transition rate matrix of the process. 

A more flexible (and more realistic) model with {\em random delays\/} was first considered in  Landriault et al.~\cite{LaReZh14}. In their setup, along with the risk reserve SNLP with trajectories of locally bounded variation, there is an independent of it sequence of i.i.d.\ random variables that serve as implementation delay times (so that for each new negative excursion of the process, there is a new independent delay time).  The authors  studied the Laplace transform of the Parisian ruin time when delays were exponentially distributed or followed Erlang mixture distributions (noting that switching from deterministic delays to stochastic ones with such distributions improves the tractability of the resulting expressions). They also studied a version of the two-sided exit problem  ``when the first passage time below level zero is substituted by the Parisian ruin time''. Frostig and Keren-Pinhasik~\cite{FrKe20} studied Parisian ruin with an ultimate bankruptcy barrier (as  in~\cite{Cz16} in the case  of deterministic delay) and i.i.d.\ exponentially (and then Erlang) distributed random delays. Baurdoux et al.~\cite{BaPaPeRe20}  studied the Gerber--Shiu distribution at Parisian ruin with exponential implementation delays in the SNLP setup.

In the present paper, we consider a natural interesting extension to the Parisian ruin problem with a risk reserve SNLP. The distinctive feature of this extension is that  the distribution of the random delay windows'  lengths can depend  on  the  deficit at the   epochs when the risk reserve process turns negative, starting a new negative excursion. This includes the possibility of an immediate ruin  when the deficit hits a certain subset. In this general setting, we derive a closed-from expression for the Parisian ruin probability and the joint Laplace transform of the Parisian ruin time and the deficit at ruin.  Our results are illustrated by an example  where the risk reserve follows the classical CL dynamics whereas the delay period distributions are finite mixtures of Erlang distributions with parameters depending on the deficit value at the beginning of the respective negative excursion.


More formally, we    assume in this paper that  $X:= \{X_t\}_{t\ge 0}$ is an SNLP   with c\`adl\`ag  trajectories   of locally   bounded variation, starting at $X_0=u\in \R$. To indicate this for different values of $u$, the respective probability and expectation symbols will be endowed with subscript $u$, as in~$\exn_u.$ The  cumulant generating function   $\psi (\theta):=\ln \exn_0  e^{\theta X_1}$ of such a  process~$X$  is clearly finite for all $\theta \ge 0$ and has the form
\begin{equation}
\psi (\theta  ) := a\theta + \int_{(-\infty, 0)} (e^{\theta x }-1) \Pi (dx), \quad \theta \ge 0,
\label{Cumul}
\end{equation}
where the measure $\Pi$ is such that $\int_{(-1, 0)} | x| \Pi (dx)<\infty.$  As well-known, this means that our process  is just a linear drift minus a pure jump subordinator (see, e.g., Section~8.1 in~\cite{Ky14}). We   also assume satisfied the standard safety loading  condition
\begin{equation}
\exn_0 X_1 >0
\label{Solvency}
\end{equation}
(clearly, $\exn_0 |X_1|<\infty$ under the above condition as $X$ is spectrally negative).

Denote by $\mathbb{F}:=\{\Fi_t\}_{t\ge 0}$ the natural filtration for~$X$. For $x, y\in \R,$ introduce the first hitting times
\begin{align*}
\tau_{x}^-   := \inf\{t>0: X_t <x \} \quad
\mbox{and}\quad
\tau_{y}^+   := \inf\{t>0: X_t >y \}.
\end{align*}
In view of~\eqref{Solvency}, $\tau_{x}^-$ is an improper random variable when $x\le X_0$. Setting $\tau_{0,0}^+  :=0,$ we further define recursively for $k=1,2,\ldots$ the following (improper, due to~\eqref{Solvency}) $\mathbb{F}$-stopping times:
\begin{align*}
\tau_{0,k}^-   := \inf\{t>\tau_{0,k-1}^+: X_t <0 \} \quad
\mbox{and}\quad
\tau_{0,k}^+   := \inf\{t>\tau_{0,k}^-: X_t >0 \}.
\end{align*}
Note that, due to \eqref{Solvency}, the   time $\tau_{0,k}^+$ is always finite on the event  $\{\tau_{0,k}^-<\infty \}$.

In words,    $\tau_{0,k}^-$  is the time when the $k$th negative excursion of the process~$X$ starts and  $\tau_{0,k}^+$ is the time when that excursion  ends. If $\tau_{0,k-1}^-<\infty$ but $\tau_{0,k}^-=\infty$ for some $k\ge 1$, then  there are only $k-1$ negative excursions of the risk reserve process.

To formally construct  random delay times,   suppose that $P_x(B)$ is a stochastic kernel on $(-\infty, 0)\times \Bi([0,\infty))$. That is, for  any fixed  $B\in \Bi([0,\infty))$, $P_{x} (B)$ is a measurable function of~$x$  and, for any fixed $x< 0$,  $P_{x} (B)$  is a probability measure in~$B\in  \Bi([0,\infty))$.  Further, let  $F_x(s):= P_x((-\infty, s]),$ $s\ge 0,$ be the distribution function of $P_x, $ $\overline{F}_x(s):=1-F_x(s) $ its right tail. Denote by
\[
F_{x}^\leftarrow (y):=\inf\{ s\ge 0: F_x (s) \ge y\}, \quad y\in (0,1),
\]
the generalized inverse of $F_x$. Note that $ F_{x}^\leftarrow (y) ,$ $(x,y)\in D:=  (-\infty, 0)\times (0,1) ,$ is a measurable function. Indeed, since $F_x (y)$ is right-continuous and non-decreasing in~$y$, for any $z\ge 0$  one has  $\{(x,y)\in D: F_{x}^\leftarrow (y)\le z\}
 = \{(x,y)\in D: F_{x} (z) -y\ge 0\}  $, which is a measurable set on the plane as both $F_{x} (z)$ and $y$ are measurable functions of~$(x,y).$

Further, let $\{U_n\}_{n\ge 1}$ be a sequence of i.i.d.\ random variables  uniformly distributed on $(0,1)$ that is  independent of the process~$X$. The length  $\eta_k $   of the $k$-th   delay window, $k=1,2,\ldots,$ is then defined on the event $\{\tau_{0,k}^-<\infty\}$   as
\[
\eta_k:= F_{\chi_k}^\leftarrow (U_k), \quad \mbox{where }\ \chi_k:= X_{\tau_{0,k}^-}
\]
(on $\{\tau_{0,k}^-=\infty\}$ we can leave both~$\chi_k$ and~$\eta_k $ undefined). Note that this allows one to model situations where    $\eta_k=0$ for some values of $\chi_k$. This happens, for instance, in   cases where   delay is only granted when the deficit $\chi_k$ is above a certain negative threshold.

We say that Parisian ruin occurs in our model if
\[
N:=
\inf \{k\ge 1: \tau_{0,k}^-<\infty, \tau_{0,k}^- + \eta_k <\tau_{0,k}^+\} <\infty,
\]
and define on the event $\{N<\infty \}$ the Parisian ruin time as
\[
T:= \tau_{0,N}^- + \eta_N.
\]

To state our results, we have to recall   the definition of the scale functions. For $q\ge 0,$ the $q$-scale  function $W^{(q)}$ for the process $X$ is defined as a  function on $\R$ such that (i)~$W^{(q)}(x)=0$ for $x<0 $  and (ii)~$W^{(q)}(x)$ is continuous on $[0,\infty)$ and 
\begin{align}
\int_{[0,\infty)} e^{-\beta x} W^{(q)}(x) dx = \frac1{\psi (\beta)-q}, \quad \beta > \Phi (q), 
\label{q-scaled1}
\end{align}
where  $\Phi (q):= \sup\{\theta \ge 0: \psi (\theta)=q\},$ $q\ge 0$
(see, e.g., Section~8.2 in~\cite{Ky14}). One refers to $W:=W^{(0)}$ as just the scale function for~$X$. Note that the $q$-scale  functions can be obtained  as the scale  functions for SNLPs with ``tilted distributions": for $q\ge 0,$ 
\begin{align}
 W^{(q)}(x) = e^{\Phi (q) x}W_{\Phi (q)}(x), \quad x\in\R,
\label{q-scaled2}
\end{align}
where $W_\nu (x)$ is the scale function for the L\'evy process with the  cumulant  function $\psi_\nu (\theta) : = \psi (\theta+\nu)- \psi (\nu)$ 
(Proposition~2 in~\cite{Su08}).

Several  important characteristics of and fluctuation identities for SNLPs can be expressed  in terms of their scale functions. In particular, the distribution $\pr_u (\chi_1\in \cdot \ ,\tau_0^-<\infty) $ of the first negative value~$\chi_1$ of the process  given $X_0=u>0$ has (defective) density $h_u(x)$ that can be written as
\begin{align*}
h_u(x)  = \int_{(0-,u]} \Pi ((-\infty,x+z-u])dW(z), \quad x<0,
 \end{align*}
see, e.g., p.\,277 in~\cite{Ky14} (note that the formula for that distribution on that page in~\cite{Ky14}  contains a typo: instead of $\Pi$ there  must be the L\'evy measure for the spectrally positive  process~$-X$). Another formula we will use below provides an expression for the ``incomplete Laplace transform" for $\tau^+_y$: for $q\ge 0$ and $t,y>0,$
\begin{align}
 \exn_0 (e^{-q\tau^+_y}; \tau^+_y \le t ) = e^{-qt}\Lambda^{(q)} (-y,t),
 \label{exp_tau}
\end{align}
where 
\[
\Lambda^{(q)} (x,t):= \int_0^\infty W^{(q)} (x+z)\frac{z}{t} \pr_0 (X_t\in dz), \quad x\in \R, \ t>0 
\]
(see, e.g., Lemma~4.2 in~\cite{LoPaSu18}); one could also compute the left-hand side of~\eqref{exp_tau} using the expression for the  distribution function of~$\tau^+_y$ provided in~\eqref{function_G}.

We also note  that finding a closed-form expression for the scale function is a non-trivial problem. A ``robust" numerical method for computing $W^{(q)}$ based on~\eqref{q-scaled1} and numerical inversion of~\eqref{q-scaled2} for $W_{\Phi (q)}$ was described  in~\cite{Su08}, whereas paper~\cite{EgYa14} presents a possible ``phase-type-fitting  approach" to approximating   scale functions and~\cite{HuKy10} presented several examples where closed form expressions for the    scale function  are  available and described a methodology for finding such expression.

Finally, we set
\begin{align}
\label{function_G}
G_y (t):=\pr_0 (\tau^+_y \le t)
= - y\frac{\partial}{\partial y }\int_0^t \pr_0 (X_s>y)\frac{ds}s, \quad y,t >0,
\end{align}
where the expression on the right-hand side comes from the celebrated Kendall's formula (see, e.g.,~\cite{BoBu01} or p.\,725 in~\cite{Bi75}, \cite{Zo64}) and let
 \begin{align}
\label{function_K}
K(x)&:= \exn_0\overline{F}_x(\tau^+_{|x|})
=\int_0^\infty \overline{F}_x(t) d G_{|x|}(t) , \quad x<0,
\\
\label{function_H}
H(v) & := \int_{-\infty}^0 K(x) h_v(x) dx,\quad  v\ge 0.
\end{align}
Our first  result is stated in the following assertion.

\begin{theo}\label{T1}
 Under the above assumptions, the  probability of no Parisian ruin when the initial reserve is $X_0=u\ge 0$ is equal to
\begin{align}
 \pr_u (N=\infty)
 & =
 \exn_0 X_1\biggl(
 W(u) +  \frac{W(0)}{1- H(0)}  H(u)\biggr).
 \label{Formula_T1}
 \end{align}
\end{theo}

One can also compute the joint Laplace transform for the Parisian ruin time and the deficit at the time of that ruin. To state our result, we need to introduce further notations. For $ v,w\ge 0$ and $x<0,$ set 
\begin{align}
M_1 ( v,w, x)& :  = 
\int_0^1 \bigl[ e^{  (\psi(w)-v) F_{x}^\leftarrow (s)+  w x} 
 -   e^{-v F_{x}^\leftarrow (s) }\Lambda^{(\psi (w))}  (x,F_{x}^\leftarrow (s))\bigr] ds ,
\label{M1}
\\
M_2 (v, x)& :  =   \exn_0  e^{-v   \tau_{|x|}^+}  \ind(\tau_{|x|}^+ \le F^{\leftarrow}_{x} (U_1)   ) 
=\int_0^1   e^{   -v  F_{x}^\leftarrow (s) }   \Lambda^{(v)}  (x,F_{x}^\leftarrow (s)) \, ds ,
\label{M2}
\end{align}
where the last equality holds true since  $\exn_0  e^{-v   \tau_{|x|}^+}  \ind(\tau_{|x|}^+ \le r  ) =e^{-vr}\Lambda^{(v)} (x,r)$ by Lemma~4.2 in~\cite{LoPaSu18} and $U_1$ is independent of~$\tau_{|x|}^+$.

Finally, assuming in addition that $u\in [0,b]$, $b>0,$ we set
	\begin{align*}
Q_1 (u,v,w) 
& := \exn_u e^{-v\tau_0^-} \ind(\tau_0^- < \tau_b^+) M_1  (v,w,\chi_1)
\\
& =
\int_0^b \int_{(-\infty,-y)} M_1 (v,w,y+\theta)
 \Big(\frac{W^{(v) }(u)W^{(v) }(b-y)}{W^{(v) }(b )}- W^{(v) }(u-y)\Big) \Pi (d\theta)dy,
 \\
 Q_2 (u,v ) 
 & := \exn_u e^{-v\tau_0^-} \ind(\tau_0^- < \tau_b^+) M_2  (v, \chi_1)
 \\
 & =
 \int_0^b \int_{(-\infty,-y)} M_2 (v, y+\theta)
 \Big(\frac{W^{(v) }(u)W^{(v) }(b-y)}{W^{(v) }(b )}- W^{(v) }(u-y)\Big) \Pi (d\theta)dy
\end{align*}
(the second equalities in both formulae   follow from the result of Exercise~10.6 on p.\,303 in~\cite{Ky14}).

\begin{theo}\label{T2}
	Under the above assumptions, for $b, v, w \ge 0  $ and $u\in [0,b],$  one has
	\begin{align*}
	\exn_u (e^{-v T + w X_T}; T<\tau_b^+)
	& =Q_1 (u,v,w) + \frac{Q_1 (0,v,w)Q_2 (u,v )}{1- Q_2 (0,v )}.
	\end{align*}
\end{theo}

\section{Proofs}


We will start with the following simple auxiliary assertions that may be well-known.

\begin{lemo}\label{L1}
Let $\xi$ and $\zeta $ be random variables on a common probability space, $\mathcal{G}$ be a sub-$\si$-algebra on that space.

{\rm (i)}~If $\exn (|\xi| + |\zeta|)<\infty$   and $\xi$ is independent of the pair $(\zeta,\mathcal{G})$ then $$\exn (\xi \zeta | \Gi)=\exn \xi\cdot  \exn (\zeta | \Gi) .$$

{\rm (ii)}~If $\zeta$ is  $\mathcal{G}$-measurable, $\xi $ is independent of $\mathcal{G}$ and has distribution function $G,$   then $$\exn (\ind (\xi > \zeta)| \Gi)= 1-G (\zeta) .$$
\end{lemo}

\begin{proof}Both statements can be verified by  straightforward computations. First observe that the right-hand sides in the above relations are clearly  $\mathcal{G}$-measurable. Second, for an arbitrary $A\in  \mathcal{G}$, in case~(i) by independence one has $\exn  \xi \zeta \ind_A = \exn \xi\exn \zeta \ind_A = \exn  \xi \exn (\exn(\zeta| \mathcal{G} )\ind_A) ,$  yielding the desired relation, whereas in case~(ii) one has
\begin{align*}
\exn  \ind (\xi > \zeta)\ind_A
 &  = \int \exn (\ind (\xi> \zeta)\ind_A |\zeta=y)\pr (\zeta\in dy)
  \\
  & = \int \exn  \ind (\xi > y)\exn (\ind_A |\zeta=y)\pr (\zeta\in dy)
  \\
  &
  =  \int (1-G(y))\exn (\ind_A |\zeta =y)\pr (\zeta\in dy) =  \exn   (1-G(\zeta))\ind_A ,
\end{align*}
which establishes the second desired relation.
\end{proof}

\begin{proof}[Proof of Theorem~\ref{T1}] Our initial step is similar to the one from~\cite{LoCzPa13}. The probability of no Parisian ruin when the initial reserve is $u>0$ equals
\begin{align}
\pr_u (N=\infty)
& =
\pr_u (\tau_0^-=\infty) + \pr_u (\tau_0^-<\infty, N=\infty)
\notag \\
& =
\pr_u (\tau_0^-=\infty) + \exn_u \exn_u \bigl(\ind (\tau_0^-<\infty)\ind (N=\infty)|\Fi_{\tau_0^-}\bigr)
\notag \\
& =
\pr_u (\tau_0^-=\infty) + \exn_u \bigl[ \ind (\tau_0^-<\infty) \exn_u\bigl(\ind (N=\infty)|\Fi_{\tau_0^-}\bigr)\bigr].
\label{Pu-1}
\end{align}
Observe that, by the strong Markov property and the absence of positive jumps, on the event $\{\tau_0^-<\infty\}$ the process 
\begin{align}
\label{tilde_X}
\widetilde X:=\{\widetilde X_t:=X_{\tau_{0,1}^+ +t}\}_{t\ge 0} 
\end{align}
is an independent of $\Fi_{ \tau_{0,1}^+}$ L\'evy process with cumulant~\eqref{Cumul} and initial value  $ \widetilde X_0=0$ (see, e.g., Theorem~3.1 in~\cite{Ky14}).
We will keep all the notations we introduced for the functionals of the process $X$ for the respective functionals of $\widetilde{X}$, endowing them with a tilde,  so that, say, $\widetilde{N}$ denotes the total number of negative excursions in $\widetilde{X}$ needed for the Parisian ruin when the risk reserve dynamics are  represented by that process  ($\widetilde{N}=\infty$ if there is no such ruin).

Now we can write that, on the event $\{\tau_0^-<\infty\},$ one has
\begin{align}
\exn_u\bigl(\ind (N=\infty)|\Fi_{\tau_0^-}\bigr)
  & =
 \exn_u\bigl(\ind (\tau_0^- + \eta_1 \ge \tau_{0,1}^+) \ind (\widetilde N=\infty)|\Fi_{\tau_0^-}\bigr)
\notag
 \\
 & =
 \exn_u\bigl[\exn_u\bigl(\ind (\tau_0^- + \eta_1 \ge  \tau_{0,1}^+)  \ind (\widetilde N=\infty)| \Fi_{ \tau_{0,1}^+}\bigr)\big|\Fi_{\tau_0^-}\bigr]
 \notag \\
 & = \pr_0 (N=\infty)
 \exn_u\bigl[\exn_u\bigl(\ind ( \eta_1\ge \tau_{0,1}^+ -\tau_0^- )  | \Fi_{ \tau_{0,1}^+}\bigr) \big| \Fi_{\tau_0^-}\bigr],
 \label{Pu-3}
\end{align}
where, to get the third equality, we used Lemma~\ref{L1}(i) with $\xi =\ind (\widetilde N=\infty) $ to re-express  the inner conditional expectation in the second line.
As $$\{\eta_1\ge \tau_{0,1}^+ -\tau_0^- \}
 =\{  F_{\chi_1}^\leftarrow (U_1)\ge \tau_{0,1}^+ -\tau_0^-\}
 = \{U_1 \ge F_{\chi_1} (\tau_{0,1}^+ -\tau_0^-) \}$$ and $ F_{\chi_1} (\tau_{0,1}^+ -\tau_0^-)$ is $\Fi_{ \tau_{0,1}^+}$-measurable, we  conclude from Lemma~\ref{L1}(ii) that
 \[
 \exn_u\bigl(\ind ( \eta_1\ge \tau_{0,1}^+ -\tau_0^- )  | \Fi_{ \tau_{0,1}^+}\bigr)
 =
 \exn_u\bigl[\ind (U_1 \ge F_{\chi_1} (\tau_{0,1}^+ -\tau_0^-) )  | \Fi_{ \tau_{0,1}^+}\bigr]
  =
  \overline{F}_{\chi_1}  (\tau_{0,1}^+ -\tau_0^-).
 \]
Now setting, for $a >0,$ 
\begin{align}
\label{hat_X}
\widehat{X}:= \{\widehat{X}_t:= X_{\tau_0^-+t}- \chi_1 \}_{t\ge 0}, \quad 
\widehat{\tau}_{a}^+:= \inf \{t>0:\widehat{X}_t  >a \},
\end{align}
we obtain that, on the event $\{\tau_0^-<\infty\},$ the second factor  in the last line of~\eqref{Pu-3} equals
\[
\exn_u\bigl(\overline{F}_{\chi_1}  (\tau_{0,1}^+ -\tau_0^-) \big| \Fi_{\tau_0^-}\bigr)
 = 
 \exn_u \bigl(\overline{F}_{\chi_1}  (\widehat{\tau}_{|\chi_1|}^+) \big|\Fi_{\tau_0^-}\bigr) 
 =  
 \exn_u \bigl(\overline{F}_{\chi_1}  (\widehat{\tau}_{|\chi_1|}^+) \big| \chi_1\bigr)= K(\chi_1) ,
\]
where, to get the last two equalities,  we used the observation that, 
on that event, by the strong Markov property, the process $\widehat{X} $   is an independent of $\Fi_{ \tau_{0 }^-}$ (and hence of $\chi_1$) L\'evy process with cumulant~\eqref{Cumul} and initial value $ \widehat X_0=0$ (recall that the function $K$ was defined  in~\eqref{function_K}). From this and \eqref{Pu-1}, \eqref{Pu-3} we derive that
\begin{align*}
\pr_u (N=\infty)
& =
\pr_u (\tau_0^-=\infty) +
\pr_0  (N=\infty)  \exn_u (K(\chi_1); \tau_0^- <\infty). 
\end{align*}

Setting now $u=0$ yields
\begin{align*}
\pr_0  (N=\infty) = \frac{\pr_0 (\tau_0^-=\infty)}{1- \exn_0 (K(\chi_1); \tau_0^- <\infty)}.
\end{align*}
Recalling that, in the case of an SNLP  with  positive  drift, one has $\pr_u (\tau_0^-=\infty)=\exn_0 X_1 W(u)$ (see, e.g., Theorem~8.1(ii) in~\cite{Ky14}), we obtain that
\begin{align*}
\pr_u (N=\infty)
& =
\exn_0 X_1\biggl[
W(u) + W(0) \frac{\exn_u (K(\chi_1); \tau_0^- <\infty)}{1- \exn_0 (K(\chi_1); \tau_0^- <\infty)}\biggr].
\end{align*}
Expressing the expectations inside the square brackets in terms of the function $H$ defined in~\eqref{function_H} yields representation~\eqref{Formula_T1}. This completes the proof of Theorem~\ref{T1}.
\end{proof}

In the proof of Theorem~\ref{T2} we will use the following observation  that may be well-known, but for which we could not find a suitable reference.

\begin{lemo}\label{L2}
	Assume that $\tau$ and $\sigma $ are  stopping times relative to a  filtration $\{\Hi_t\}_{t\ge 0}.$ Then   $ \{\tau<\sigma\}\in\Hi_\tau.$
\end{lemo}

\begin{proof}
For $t\ge 0,$ we have
\begin{align*}
\{\tau<\sigma\}\cap\{\tau \le t\} & =  \{\tau<\sigma,\tau \le t, \sigma > t\}  \cup \{\tau<\sigma,\tau \le t, \sigma  \le t\} \\
& =
 \{ \tau \le t, \sigma > t\}  \cup \{\tau<\sigma,  \sigma  \le t\}.
\end{align*}
The first event in the union in the second line is clearly in $\Hi_t$, whereas for the second one we have
\[
\{\tau<\sigma,  \sigma  \le t\}
 = \bigcup_{r\in \mathbb{Q},\ r<t} ( \{\tau\le r\} \cap \{r<\sigma \le t\}),
\]
where obviously $\{\tau\le r\}\in \Hi_r\subseteq \Hi_t$ and $\{r<\sigma \le t\}\in \Hi_t$ when $r< t$. Lemma~\ref{L2} is proved.
\end{proof}

\begin{proof}[Proof of Theorem~\ref{T2}] Our starting point is to observe that, for $u, v, w\ge 0,$ one has 
\begin{align}
\exn_u (e^{-v T + w X_T}  &; T<\tau_b^+) =
\exn_u \exn_u  (e^{-v T + w X_T}\ind(T<\tau_b^+)\ind(\tau_0^-<\tau_b^+) |\Fi_{\tau_0^-} )
\notag
\\
& =
\exn_u \bigl[e^{-v \tau_0^-} \ind(\tau_0^-<\tau_b^+) \exn_u  (e^{-v (T -\tau_0^-)+ w X_T}\ind(T<\tau_b^+) |\Fi_{\tau_0^-} )\bigr],
\label{Euu}
\end{align}	
where the second equality follows from Lemma~\ref{L2}. The conditional expectation in the second line is equal to $E_1 + E_2$, where
\begin{align*}
E_1 & := \exn_u  (e^{-v (T -\tau_0^-)+ w X_T}\ind(T<\tau_b^+)\ind(N=1)  |\Fi_{\tau_0^-} ),
\\
E_2& := \exn_u  (e^{-v (T -\tau_0^-)+ w X_T}\ind(T<\tau_b^+)\ind(N>1)  |\Fi_{\tau_0^-} ).
\end{align*}
First we will evaluate $E_1.$ On the event $\{\tau_0^-<\tau_b^+, N=1\}= \{ \tau_0^- <\tau_b^+, \eta_1 < \widehat{\tau}^+_{|\chi_1|}\} $ (see~\eqref{hat_X}), one has $T= \tau_0^- + \eta_1$, $X_T= X_{\tau_0^- + \eta_1} =   \chi_1 + \widehat{X}_{\eta_1}$  and automatically $ T<\tau_b^+ $ (as the Parisian ruin occurs during the first negative excursion and that excursion started prior to time~$\tau_b^+$). Therefore, on  the event $\{\tau_0^-<\tau_b^+\}\in \Fi_{\tau_0^-}$, one has 
\begin{align*}
E_1 & =  \exn_u  (e^{-v \eta_1 + w (\widehat{X}_{\eta_1} +\chi_1) } \ind(\tau_0^-<\tau_b^+) \ind(N=1)  |\Fi_{\tau_0^-} )
 \\
 & =  e^{ w  \chi_1  } 
  \exn_u  (e^{-v \eta_1 + w \widehat{X}_{\eta_1}   } \ind(\eta_1 <\widehat{\tau}^+_{|\chi_1|})    |\Fi_{\tau_0^-} ).
\end{align*}
On the event $\{\tau_0^-<\tau_b^+\} $ the process  $\widehat{X}$ is  an independent of $\Fi_{\tau_0^-}$ distributional copy of $X$ (cf.\ our comment after~\eqref{hat_X}), so that the only random component inside the conditional expectation in the second line that is not independent of $ \Fi_{\tau_0^-}$ is $\chi_1$ (it participates in  both  $\eta_1$ and $\widehat{\tau}^+_{|\chi_1|}$). We conclude that, on the event $\{\tau_0^-<\tau_b^+\} $, that conditional expectation equals 
\begin{align*}
\exn_u  (e^{-v \eta_1 + w \widehat{X}_{\eta_1}   } \ind(\eta_1 <\widehat{\tau}^+_{|\chi_1|})    |\chi_1 )
 = 
\exn_u \bigl[e^{-v \eta_1 } \exn_u  (e^{  w \widehat{X}_{\eta_1}   } \ind(\eta_1 <\widehat{\tau}^+_{|\chi_1|})    |\chi_1, \eta_1 )\big|\chi_1\bigr].
\end{align*}
Given that $\chi_1= x<0,$ $\eta_1=t>0,$ the inner conditional expectation on the right-hand side of the above formula is equal to $\exn_0   e^{  w X_t   } \ind(t < \tau^+_{|x|})$. This expression  can be computed similarly to the argument used in the proof of Lemma~4.3 in~\cite{LoPaSu18}: 
\begin{align*}
\exn_0   e^{  w X_t   } \ind(  \tau^+_{|x|}>t) 
& = \exn_0   e^{  w X_t   }  
- \exn_0   e^{  w X_t   } \ind(  \tau^+_{|x|}\le t) 
\notag\\
& = e^{t \psi(w)} - \int_{(0,t]} 
 \exn_0  ( e^{  w X_t   } | \tau^+_{|x|}=s) \pr_0 (\tau^+_{|x|}\in ds)
\notag \\
& = e^{t \psi(w)} - e^{w|x|}\int_{(0,t]} 
 \exn_0  ( e^{  w (X_t -X_s)   } | \tau^+_{|x|}=s) \pr_0 (\tau^+_{|x|}\in ds)
\notag \\
& = e^{t \psi(w)} - e^{-wx  +t \psi(w) }\int_{(0,t]}  e^{-s \psi(w)} 
 \pr_0 (\tau^+_{|x|}\in ds),
 \\
 & =  e^{t \psi(w)} - e^{-wx} \Lambda^{(\psi (w))} ( x,t),
  \label{psi_int}
\end{align*}
where we used the spectral negativity of $X$  and the strong Markov property to get the third and fourth  equalities and representation~\eqref{exp_tau} to get the fifth one. Combining these computations, we obtain that, on the event $\{\tau_0^-<\tau_b^+\} ,$  one has 
\begin{align*}
E_1 & =   e^{ w  \chi_1  } \exn_u [ e^{-v\eta_1}
( e^{  \psi(w) \eta_1} - e^{-w\chi_1} \Lambda^{(\psi (w))}  (\chi_1,\eta_1))|\chi_1]
\\
& =    \exn_u   
( e^{  (\psi(w)-v) \eta_1 +  w  \chi_1} -  e^{-v\eta_1} \Lambda^{(\psi (w))}  (\chi_1,\eta_1) |\chi_1) = M_1 ( v,w,\chi_1), 
\end{align*}
where   $M_1 ( v,w, x)$ was introduced in~\eqref{M1}.

Now we will turn to the term~$E_2$. On the event $\{T<\tau_b^+\},$  relation $N>1$ is equivalent to $\tau_0^-+\eta_1 \ge \tau_{0,1}^+$, so that on the event $\{\tau_0^-<\tau_b^+\}$ one has 
\begin{align*}
E_2&  = \exn_u  (e^{-v (T -\tau_0^-)+ w X_T}\ind(T<\tau_b^+)\ind(\tau_0^-+\eta_1 \ge \tau_{0,1}^+)  |\Fi_{\tau_0^-} )
\\
& = \exn_u  \big[e^{-v (\tau_{0,1}^+ -\tau_0^-) }\ind(\tau_0^-+\eta_1 \ge \tau_{0,1}^+) 
\exn_u (e^{-v (T -\tau_{0,1}^+   ) + w  {X}_T} \ind(T<\tau_b^+)|\Fi_{\tau_{0,1}^+},\eta_1 )\big|\Fi_{\tau_0^-} \big]
\\
& 
= \exn_u  \big[e^{-v (\tau_{0,1}^+ -\tau_0^-) }\ind(\tau_0^-+\eta_1 \ge \tau_{0,1}^+) 
\exn_u (e^{-v \widetilde{T} + w \widetilde{X}_{\widetilde{T}}} \ind(\widetilde T<\widetilde \tau_b^+)|\Fi_{\tau_{0,1}^+},\eta_1 )\big|\Fi_{\tau_0^-} \big],
\end{align*}
where we used the process $\widetilde X$ from~\eqref{tilde_X} (and the relevant to it random times $\widetilde T, \widetilde \tau_b^+$) and the observation that the relation $T<\tau_b^+$ is equivalent to $ \widetilde T<\widetilde \tau_b^+ $ provided that $ \tau_0^-<\tau_b^+.$ Using the strong Markov property and the fact that $\widetilde X_0=0$, we conclude that 
\begin{align*}
E_2&  =   \exn_u  \bigl(e^{-v (\tau_{0,1}^+ -\tau_0^-) }\ind(\tau_0^-+\eta_1 \ge \tau_{0,1}^+) 
  \big|\Fi_{\tau_0^-} \bigl)
  \exn_0 (e^{-v  {T} + w  {X}_{T}} ;   T<  \tau_b^+)
  \\
  & = \exn_u  \bigl(e^{-v  \widehat\tau_{|\chi_1|}^+}  \ind(F^{\leftarrow}_{\chi_1} (U_1) \ge \widehat\tau_{|\chi_1|}^+) 
  \big|\chi_1 \bigr)
  \exn_0 (e^{-v  {T} + w  {X}_{T}} ;   T<  \tau_b^+)
  \\
  & = M_2 (v,  \chi_1) \exn_0 (e^{-v  {T} + w  {X}_{T}} ;   T<  \tau_b^+),
\end{align*}
where $M_2 (v, x) $ was introduced in~\eqref{M2}.

Substituting the computed values for $E_1$ and $E_2$ into~\eqref{Euu} yields 
\begin{align*} 
\exn_u ( & e^{-v T + w X_T}   ; T<\tau_b^+)  
\\
& =   
\exn_u \bigl[e^{-v \tau_0^-} \ind(\tau_0^-<\tau_b^+) (M_1 (v,w,\chi_1)+ M_2 (v,\chi_1)\exn_0 (e^{-v  {T} + w  {X}_{T}} ;   T<  \tau_b^+) )\bigr]
\\
& = Q_1 (u,v,w) +  Q_2 (u,v )\exn_0 (e^{-v T + w X_T}    ; T<\tau_b^+) .
\end{align*}
Setting $u=0$ we recover $\exn_0 (e^{-v T + w X_T}    ; T<\tau_b^+)$ as 
 $Q_1 (0,v,w)/ (1-  Q_2 (0,v )).$ Substituting this back into the above formula completes the proof of Theorem~\ref{T2}. 
\end{proof}

\section{Examples}

Consider the classical CL model:
\[
X_t = X_0 + ct -\sum_{j=1}^{A_t} \xi_j, \quad t\ge 0, 
\]
where  $c>0$ is a constant premium payment rate and the  Poisson claims arrival process $\{A_t\}_{t\ge 0 }$ with rate $\lambda >0$ is independent of the  sequence of i.i.d.\ exponentially distributed claim sizes $\{\xi_n\}_{n\ge 1}$ with rate $\al >0.$

Clearly, in this case one  has $\psi (\theta)= c\theta +\la (\frac{\al}{\al+\theta}-1),$ $\theta >-\al,$  so that condition~\eqref{Solvency} turns into 
\[
\exn_0 X_1 = c- \la / \al >0.
\]
Elementary computation yields 
\[
\Phi (q) = \frac1{2c} \Bigl(\sqrt{(\al c -\la - q)^2 + 4 q\al  c} - (\al c -\la - q)\Bigr), \quad q\ge 0,
\]
and 
\begin{align*}
W(x) = \frac{\al}{\al c- \la }\Bigl(1- \frac{\la }{\al c} 
 e^{-(\al - \la /c)x}\Bigr)\ind (x\ge 0), \quad \mbox{with \ } W(0)= \frac1c 
\end{align*}
(see p.\,251 in~\cite{Ky14}). Note that an explicit expression (in the form of an infinite series) is also available for the $q$-scale function $W^{{(q)}}$ for this model (Example 5.3 in~\cite{BeOe20}). 

It is well-known that, for this model, one has 
\begin{align}
\label{Prob_ruin_CPP}
\pr_u(\tau_0^-<\infty) = \frac{\la}{\al c}e^{- (\al - \la/c)u}, \quad u\ge 0
\end{align} 
(see, e.g., p.\,78 in~\cite{AsAl10}). Due to the memoryless property of the exponential distribution, the conditional distribution of $-\chi_1$ (given that~$X$ ever turns negative) will coincide with the distribution of~$\xi_1$, so that 
\[
\pr_u (\chi_1 \le x ,  \tau_0^-<\infty  )
 = \pr_u (\chi_1 \le x \,  |\, \tau_0^-<\infty  )\pr_u (\tau_0^-<\infty)
  =\frac{\la}{\al c}e^{\al x- (\al - \la/c)u},\quad x\le 0. 
\]
Therefore 
\[
h_u (x) = \frac{\la}{ c}e^{\al x- (\al - \la/c)u},\quad x< 0,
\]
and hence
\begin{align}
\label{H_exp}
H(v) & = \int_{-\infty}^0 K(x) h_v (x) \, dx   
= H(0) e^{ - (\al - \la/c)v}, \quad 
H(0)=  \frac{\la}{ c}  \int_{-\infty}^0 e^{\al x} K(x)  dx   .
\end{align}
Substituting the obtained expressions into~\eqref{Formula_T1} yields
\begin{align}
\label{Paris_ru_CL}
\pr_u (N<\infty) = \frac{\la}{\al c} \biggl[1 -   \frac{(\al c-\la) H(0)}{\la (1-H(0))}\biggr]
e^{- (\al - \la/c)u}, \quad  u\ge 0.
\end{align}
Comparing this with~\eqref{Prob_ruin_CPP}, we see that, for the CL risk reserve process model, the Parisian ruin probability differs from the ``usual'' one by having a smaller constant factor in front of the same exponential term.

To compute the value of $H(0)$ in~\eqref{Paris_ru_CL}, we need to specify the   distribution  of the delay window length. We will consider two special cases. 

\medskip 

{\em Case 1.} Assume that the conditional distribution of the window length is exponential with parameter depending on the deficit: there is a Borel function $r:(-\infty, 0)\to (0,\infty]$ such that $\overline{F}_x (t) =e^{-r (x)t},$ $t>0 $ (where the value $r(x)=\infty$ means immediate ruin when  $\chi_1$ is equal to that~$x$). Then, by Theorem~3.12 in~\cite{Ky14},
\begin{align}
\label{KPhi}
K(x) & = \int_0^\infty \overline{F}_x (t)dG_{|x|}(t)
  = \exn_0 e^{-r (x) \tau^+_{|x|}}  = e^{  \Phi (r(x))x}, \quad x<0,
\end{align}
and hence 
\begin{align}
\label{H_exp0}
H(0)=  \frac{\la}{ c}  \int_{-\infty}^0 e^{ [\al + \Phi (r(x))]x} dx   .
\end{align}
This  quantity   can be explicitly evaluated, for instance, in the special case when $r(x)$ is piece-wise constant:
\[
r(x):= \sum_{k=1}^n r_k \ind (x\in (a_{k-1}, a_k])
\] 
for some  $n\ge 1,$   $r_k\in (0,\infty],$ $k=1,\ldots, n,$ and  $-\infty=:a_0 <a_1 < \cdots < a_{n-1}<a_n:=0.$  Then~\eqref{H_exp0} turns into 
\begin{align*}
H(0) &   
=  \frac{\la}{ c} \sum_{k=1}^n \int_{a_{k-1}}^{a_k} e^{ (\al + \Phi (r_k))x} dx  =     \frac{\la}{ c}\sum_{k=1}^n \frac{e^{ (\al + \Phi (r_k))a_{k}}- e^{ (\al + \Phi (r_k))a_{k -1}}}{ \al + \Phi (r_k)},
\end{align*} 
the terms in the sum with $r_k=\infty$ being equal to~0. This example admits a straightforward extension to the case when $F_x,$ $x<0,$ are hyperexponential distributions.

\medskip 

{\em Case 2.} Assume now that the conditional distribution of the window length is a finite mixture of Erlang distributions with  parameters depending on the deficit: for an $m\ge 1,$ there are  Borel functions $p_j :(-\infty, 0)\to [0,1],$ $\sum_{j=1}^m p_j (x) \equiv 1,$ $r_j :(-\infty, 0)\to (0, \infty],$ and $\nu_j (x) :(-\infty, 0)\to \Na,$ $j = 1, \ldots, m,$ such that, for $x<0,$  
\[
\overline{F}_x (t) = \sum_{j=1}^m p_j (x) \sum_{\ell =0}^{\nu_j(x)-1} \frac{(r_j(x) t)^\ell}{\ell!} e^{- r_j (x) t},\quad t>0,
\]
is the right distribution tail of  a mixture of (up to) $m$ components that are Erlang distributions with the respective shape and rate parameters $\nu_j(x),$ $r_j(x),$ $j=1,\ldots, m.$ 
Such mixtures form a rather large class: it is well-known to be everywhere dense in the  weak convergence topology in the class of continuous distributions on $(0,\infty)$ (see, e.g., p.\,153 in~\cite{Ti94}). 

In this case, 
\begin{align*}
K(x) & = \sum_{j=1}^m p_j (x) \sum_{\ell =0}^{\nu_j(x)-1} \frac{ r_j^\ell (x)   }{\ell!} \int_{0}^\infty  t ^\ell e^{- r_j (x) t} dG_{|x|} (t) \\
& = 
\sum_{j=1}^m p_j (x) \sum_{\ell =0}^{\nu_j(x)-1}   r_j^\ell(x)  \phi_\ell (r_j (x),x)  ,
\end{align*}
where we used the fact that, by~\eqref{KPhi} and the well-known property of Laplace transforms, 
\begin{align}
\label{Kx}
\int_{0}^\infty  t ^\ell e^{- r  t} dG_{|x|} (t)
 = \ell! \phi_\ell (r,x), \qquad  \phi_\ell (r,x):= \frac{(-1)^\ell}{\ell!} \frac{\partial^\ell}{\partial r^\ell}e^{\Phi (r)x}.
\end{align}
As in Case~1, we now assume that the functions participating in the definition of $\overline{F}_x$ are piece-wise constant. Namely, there exist  $-\infty=:a_0 <a_1 < \cdots < a_{n-1}<a_n:=0$ such that whenever the deficit at the beginning of a negative excursion of the risk reserve process  hits the interval $(a_{k-1}, a_k],$ the applicable delay window length will have one and the same distribution given by a finite mixture of Erlang distributions. More formally, for some 
  $p_{j,k}\in [0,1]$ ($\sum_{j=1}^m p_{ j,k}=1$),  $r_{j,k}\in (0,\infty],$   and $\nu_{j,k}\in\Na,$     one has $r_j(x):= \sum_{k=1}^n r_{j,k} \ind (x\in (a_{k-1}, a_k]),$ $ p_j(x):= \sum_{k=1}^n p_{j,k} \ind (x\in (a_{k-1}, a_k])$ and 
$ \nu_j(x):= \sum_{k=1}^n \nu_{j,k} \ind (x\in (a_{k-1}, a_k]).$ Then from~\eqref{H_exp} and~\eqref{Kx} we get the following expression that can be evaluated for any set of the model parameters: 
\[
H(0) = \frac{\la}c \sum_{k=1}^n
\sum_{j=1}^m p_{j,k} \sum_{\ell =0}^{\nu_{j,k}-1}   r_{j,k}^\ell \int_{a_{k-1}}^{a_k}e^{\al x} \phi_\ell (r_{j,k} ,x)    dx.
\]


\end{document}